\documentclass[11pt]{amsart}

\usepackage{latexsym}
\usepackage{amssymb}
\usepackage{amsmath}
\usepackage{amsfonts}
\usepackage[mathscr]{eucal}
\usepackage{times}
\usepackage{MnSymbol}
\usepackage{color}

\usepackage{amscd}
\usepackage[all,cmtip]{xy}

\usepackage{graphicx}
\usepackage{epsfig}

\usepackage{amsthm}
\usepackage[pdfstartview=FitH]{hyperref}
\usepackage{pst-grad}
\usepackage{pst-plot}
\usepackage{comment}
\usepackage{psfrag}
\usepackage{verbatim}
\usepackage{soul}
\usepackage{lineno} 

\newtheorem{thm}{Theorem}[section]

\newtheorem{theorem}{Theorem}

\newtheorem{corollary}[theorem]{Corollary}
\newtheorem{conjecture}[theorem]{Conjecture}

\theoremstyle{definition}

\newtheorem{question}[thm]{Question}
\newtheorem{rem}[thm]{Remark}

\newtheorem{case}{Case}[section]
\newtheorem{subcase}{Case}[case]

\newtheorem*{ack}{Acknowledgements}

\numberwithin{equation}{section}

\DeclareMathOperator{\Susp}{Susp}
\DeclareMathOperator{\curv}{curv}

\DeclareMathOperator{\diam}{diam}
\DeclareMathOperator{\vol}{vol}

\newcommand{\Real}{\mathrm{\mathbb{R}}}

\newcommand{\bpt}{\mathrm{B(pt)}}
\newcommand{\bstwo}{\mathrm{B}(S_2)}
\newcommand{\bsfour}{\mathrm{B}(S_4)}
\newcommand{\bptwo}{\mathrm{B}(\mathbb{R}P^2)}

\newcommand{\csum}{\operatornamewithlimits{\#}}

\begin{document}



\title[On topological rigidity of Alexandrov $3$-spaces]{On topological rigidity of Alexandrov $3$-spaces 
}

\author[N.~B\'arcenas]{No\'e~B\'arcenas $^\dagger$}
\author[J.~N\'u\~nez-Zimbr\'on]{Jes\'us N\'u\~nez-Zimbr\'on $^\ddagger$}

\address[N.~B\'arcenas ]{Centro de Ciencias Matem\'aticas, UNAM, Campus Morelia, 59089, Morelia, Michoac\'an, M\'exico}
\email{barcenas@matmor.unam.mx}

\address[J.~N\'u\~nez-Zimbr\'on]{Centro de Ciencias Matem\'aticas, UNAM, Campus Morelia, 59089, Morelia, Michoac\'an, M\'exico}
\email{zimbron@matmor.unam.mx}


\thanks{$^\dagger$ Supported  by  CONACYT- foundational  Research  Grant 250747, PAPIIT  Grant IA 100117. $^\ddagger$ Supported by a DGAPA-UNAM postdoctoral fellowship.   }


\date{\today}


\subjclass[2010]{53C23, 57S15, 57S25}
\keywords{Collapse, Alexandrov space, Borel conjecture}


\begin{abstract}
In this note we prove the Borel Conjecture for closed, irreducible and sufficiently collapsed three-dimensional Alexandrov spaces. We also pose several questions  related  to  the  characterization  of  fundamental  groups  of three-dimensional Alexandrov spaces, finite  groups  acting  on  them  and rigidity results. 
\end{abstract}

\maketitle



\section{Introduction and results}
\label{section:introduction}

Alexandrov spaces are inner metric spaces which admit a lower sectional curvature bound in a synthetic sense. They constitute a generalization of the class of complete Riemannian manifolds with a lower sectional curvature bound and since their introduction they have proven to be a natural setting to address geometric-topological questions of a global nature. Therefore, a central problem is to determine whether what is already known in the smooth or topological settings still holds in Alexandrov geometry. 

Regarding topological rigidity of spaces, an important conjecture originally formulated in the topological manifold category, is the \textit{Borel conjecture}. It asserts that if two closed, aspherical $n$-manifolds are homotopy equivalent, then they are homeomorphic. The proof of this conjecture in the three-dimensional case is a consequence of Perelman's resolution of the Geometrization Conjecture \cite{Por}.

On the  other  hand, in  high dimension (meaning greater or  equal than  five),  the  Borel Conjecture  for  an  aspherical manifold with  fundamental  group $G$ is  consequence  of the  Farrell-Jones  Conjecture in  Algebraic $K$- and $L$-Theory for  the  group $G$ \cite{LuRe}. A  lot  of  effort  in geometric topology and  surgery  theory  has  been  devoted  to prove  the  Borel  conjecture in  many  cases by these methods, which rely on  transversality  arguments which  are  not  available  for  the  study  of topological  rigidity  of  low  dimensional manifolds.   

 We  will  present  in the  following  note  a  series  of  questions related  to generalizations of the Borel conjecture outside of the manifold category. Steps  in  this  direction  have been obtained,  for example, for  $\mathrm{CAT}(0)$-spaces as  a  consequence  of  the  Farrell Jones-Conjecture \cite{BarLuc}, and  in  another  direction  for certain classes of topological orbifolds \cite{TakYok} as  a  consequence  of  classification  efforts  in three  dimensional geometric  topology beyond  manifolds.
 
   Negative  results concerning  the   topological rigidity  of  singular  spaces  of  geometric  nature,  such as  the   Coxeter  complex have been obtained in  \cite{Star}. 
  
  It is therefore natural to inquire whether the Borel conjecture is still valid for closed Alexandrov $3$-spaces (cf. \cite[Remark 3.12]{GalSurvey}). These spaces are either topological $3$-manifolds or are homeomorphic to quotients of smooth $3$-manifolds by smooth orientation-reversing involutions with only isolated fixed points (see \cite{GalGui}).

In this article we address the validity of the Borel Conjecture for the class of \textit{sufficiently collapsed} and \textit{irreducible} closed Alexandrov $3$-spaces. The definition of \textit{irreducibility} for a closed Alexandrov $3$-space was introduced in \cite{GalGuiNun}. Let us recall that a closed Alexandrov $3$-space $X$ is \textit{irreducible} if every embedded $2$-sphere in $X$ bounds a $3$-ball and, in the case that the set of topologically singular points of $X$ is non-empty, it is further
required that every $2$-sided $\mathbb{R}P^2$ bound a $K(\mathbb{R}P^2)$, a cone over a real projective plane $\mathbb{R}P^2$. The condition related to collapse is described more precisely by considering the class of spaces $\mathcal{A}^3(D,\varepsilon)$, defined as the class of closed Alexandrov $3$-spaces with $\curv\geq -1$, satisfying that $\diam \leq D$ and $\vol< \varepsilon$ for given $D,\varepsilon>0$. We say that a closed Alexandrov $3$-space $X$ is sufficiently collapsed if $X\in \mathcal{A}^3(D,\varepsilon)$ for a sufficiently small $\varepsilon$ with respect to $D$. Our main result is the following.

\begin{theorem}
\label{thm:Borel_Alex_collapsing}
 For any $D>0$, there exists $\varepsilon=\varepsilon(D)>0$ such that, if $X,Y\in \mathcal{A}^3(D,\varepsilon)$ are aspherical and irreducible, then the Borel Conjecture holds for $X$ and $Y$, that is, if $X$ is homotopy equivalent to $Y$ then $X$ is homeomorphic to $Y$. 
\end{theorem}

We point out that a related result was obtained in \cite[Theorem 6.1]{Nun} where the second named author proved the Borel Conjecture for closed Alexandrov $3$-spaces admitting an isometric circle action. The proof of Theorem \ref{thm:Borel_Alex_collapsing} is based on two points: the Borel conjecture in the $3$-manifold case and the following result.

\begin{theorem}
\label{thm:collapsed_aspherical_implies_manifold}
For any $D>0$, there exists $\varepsilon=\varepsilon(D)>0$ such that, if $X\in \mathcal{A}^3(D,\varepsilon)$ is irreducible and aspherical, then $X$ is homeomorphic to a $3$-manifold.
\end{theorem}

The classification of closed collapsing Alexandrov $3$-spaces due to Mitsuishi-Yamaguchi is a key tool in the proof of Theorem \ref{thm:collapsed_aspherical_implies_manifold}. The classification of closed Alexandrov $3$-spaces admitting isometric (local) circle actions \cite{GalNun}, \cite{Nun} obtained by Galaz-Garc\'ia and the second named author also plays a role. The strategy of proof resembles that of \cite[Theorem A]{GalGuiNun}. In fact, without assuming that the spaces in question are sufficiently collapsed or irreducible, the analysis of Section \ref{ssec:1-dim-limit} implies the following result.

\begin{corollary}
Let $X$ be a closed, non-orientable Alexandrov $3$-space with fundamental group $\mathbb{Z}$, $\mathbb{Z}\rtimes \mathbb{Z}$ or $\mathbb{Z}\oplus \mathbb{Z}$. Then $\pi_2(X)\neq 0$. In particular, $X$ is not aspherical. 
\end{corollary}

For arbitrary dimension we observe the following result, which is an immediate consequence of a result of Mitsuishi (see \cite[Corollary 5.7]{Mit} and Theorem \ref{thm:Mitsuishi_result} below) and Theorem \ref{thm:Hopf} (stated below).

\begin{corollary}
\label{thm:compact_univ_cover_aspherical_implies_manifold} 
Let $X$ be a closed, aspherical Alexandrov $n$-space such that its universal cover is compact. Then $X$ is homeomorphic to a closed $n$-manifold. 
\end{corollary}

In light of Theorem \ref{thm:collapsed_aspherical_implies_manifold} and Corollary  \ref{thm:compact_univ_cover_aspherical_implies_manifold} we propose the following natural conjecture (cf. Remark \ref{conjecture}).

\begin{conjecture}\label{conj:alex-manifold}
Every closed, aspherical Alexandrov $3$-space is a $3$-manifold
\end{conjecture}

The organization of the article is the following. In Section \ref{section:preliminaries}, we briefly recall the basic structure of Alexandrov $3$-spaces following Galaz-Garc\'ia-Guijarro \cite{GalGui} and the classification of collapsing Alexandrov $3$-spaces of Mitsuishi-Yamaguchi \cite{MitYam}. In Section \ref{section:borel_conjecture_collapsing_alexandrov_3_spaces}, we prove Theorem \ref{thm:collapsed_aspherical_implies_manifold} which yields as a consequence the validity of Theorem \ref{thm:Borel_Alex_collapsing}. 
In section \ref{section:groups_spaces_questions}, we  state  some questions  related  to the  fundamental  groups  and groups  which  can  act on Alexandrov  $3$-spaces,  in  analogy  with  similar results  obtained   in  connection  with the  topological  rigidity  of  aspherical  manifolds.  


\begin{ack} The authors warmly thank Fernando Galaz-Garc\'ia and Luis Guijarro for very helpful conversations. JNZ also thanks Bernardo Villarreal and \'Angel Zaldivar for useful communications.  
\end{ack}

\section{Preliminaries}
\label{section:preliminaries}

We will assume that the reader is familiar with the general theory of Alexandrov spaces of curvature bounded below and refer to \cite{BurBurIva} for a more detailed introduction. In this section we will briefly recall some results concerning the structure of closed Alexandrov $3$-spaces. All spaces are considered to be connected throughout the article. 

\subsection{Alexandrov $3$-spaces}
 Let $X$ denote a closed Alexandrov $3$-space and for each $x\in X$, let $\Sigma_xX$ be the space of directions at $x$. The space $\Sigma_xX$ is a closed Alexandrov $2$-space with $\curv \Sigma_xX\geq 1$ (see \cite[Theorem 10.8.6]{BurBurIva}). This implies, via the Bonnet-Myers Theorem (see \cite[Theorem 10.4.1]{BurBurIva}) and the classification of closed surfaces that the homeomorphism type of $\Sigma_xX$ is that of a sphere $\mathbb{S}^2$ or that of a real projective plane $\mathbb{R}P^2$. A point $x\in X$ such that $\Sigma_xX$ is homeomorphic to $\mathbb{S}^2$ is called \textit{topologically regular}, while a point such that $\Sigma_xX$ is homeomorphic to $\mathbb{R}P^2$ is called \textit{topologically singular}. We let $S(X)$ be the subset of $X$ consisting of topologically singular points. Then $X\setminus S(X)$ is open and dense in $X$ (see \cite[Theorem 10.8.5]{BurBurIva}). Furthermore, the Conical Neighborhood Theorem of Perelman \cite{Per} states that each point $x\in X$ has a neighborhood which is pointed-homeomorphic to the cone over $\Sigma_xX$. As a consequence, $S(X)$ is a finite set. 

Topologically, a closed Alexandrov $3$-space $X$ can be described as a compact $3$-manifold $M$ having a finite number of $\mathbb{R}P^2$-boundary components with a cone over $\mathbb{R}P^2$ attached on each boundary component. In the case that $S(X)\neq \emptyset$ there is an alternative topological description of $X$ as  quotient of a closed, orientable, topological $3$-manifold $\tilde{X}$ by an orientation-reversing involution $\iota:\tilde{X}\to\tilde{X}$ having only isolated fixed points. The $3$-manifold $\tilde{X}$ is called the \textit{branched orientable double cover of $X$} (see \cite[Lemma 1.7]{GalGui}). It is possible to lift the Alexandrov metric on $X$ to an Alexandrov metric on $\tilde{X}$ having the same lower curvature bound in such a way that $\iota$ is an isometry. In particular, $\iota$ is equivalent to a smooth involution on $\tilde{X}$ regarded as a smooth $3$-manifold (a detailed description of this construction can be found in \cite[Lemma 1.8]{GalGui}, \cite
[Section 2.2]{DenGalGuiMunn} and \cite[Section 5]{GroWil}).

\subsection{Collapsing Aleandrov $3$-spaces} 

Let $\{X_i\}_{i=1}^{\infty}$ be a sequence of closed Alexandrov $3$-spaces with diameters uniformly bounded above by $D>0$ and $\curv X_i\geq k$ for some $k\in\mathbb{R}$. Gromov's Precompactness Theorem implies that (possibly after passing to a subsequence), there exists an Alexandrov space $Y$ with diameter bounded above by $ D$ and $\curv Y \geq k$ such that $X_i \stackrel {_{GH}}{\longrightarrow}Y$. In the case in which $\dim Y<3$, the sequence $X_i$ is said to \textit{collapse to $Y$}. Similarly, a closed Alexandrov $3$-space $X$ is a \textit{collapsing Alexandrov $3$-space} if there exists a sequence of Alexandrov metrics $d_i$ on $X$ such that the sequence $\{(X,d_i)\}_{i=1}^{\infty}$ is a collapsing sequence. 

The topological classification of closed collapsing Alexandrov $3$-spaces was obtained by Mitsuishi-Yamaguchi in \cite{MitYam}. We now give a brief summary of the classification. We denote the boundary of an Alexandrov space $Y$ by $\partial Y$.

In the case in which $\dim Y=2$ (cf. \cite[Theorems 1.3, 1.5]{MitYam}), for sufficiently big $i$, $X_i$ is homeomorphic to a \textit{generalized Seifert fibered space} $\mathrm{Seif}(Y)$ (see \cite[Definition 2.48]{MitYam}). In the case in which $\partial Y\neq \emptyset$, $\mathrm{Seif}(Y)$ is attached with a finite number of \textit{generalized solid tori and Klein bottles} (see  \cite[Definition 1.4]{MitYam}). 

In the event that $\dim Y=1$ and $\partial Y=\emptyset$ (cf. \cite[Theorem 1.7]{MitYam}), for big enough $i$,  $X_i$ is homeomorphic to an $F$-fiber bundle over $\mathbb{S}^1$, where $F$ is homeomorphic to one of the spaces $T^2$, $K^2$, $\mathbb{S}^2$ or $\mathbb{R}P^2$. On the other hand, if $\partial Y\neq \emptyset$ (cf. \cite[Theorem 1.8]{MitYam}), $X_i$ is homeomorphic to a union of two spaces $B$ and $B'$ with one boundary component, glued along their homeomorphic boundaries, where $\partial B$ is one of the spaces $T^2$, $K^2$, $\mathbb{S}^2$ or $\mathbb{R}P^2$. The pieces $B$ and $B'$ are determined as follows:
\begin{itemize}
\item[(i)] If $\partial B\cong \mathbb{S}^2$ then $B$ and $B'$ are homeomorphic to one of: a $3$-ball $D^3$, a $3$-dimensional projective space with the interior of a $3$-ball removed $\mathbb{R}P^3\setminus \mathrm{int}D^3$ or $\bstwo$, a space homeomorphic to a small metric ball of an $\mathbb{S}^2$-soul of an open non-negatively curved Alexandrov space $L(\mathbb{S}^2;2)$ (cf. \cite[Corollary 2.56]{MitYam}).
\item[(ii)] If $\partial B\cong \mathbb{R}P^2$ then $B$ and $B'$ are homeomorphic to a closed cone over a projective plane $K_1(\mathbb{R}P^2)$.
\item[(iii)] If $\partial B\cong T^2$ then $B$ and $B'$ are homeomorphic to one of $\mathbb{S}^1\times D^2$, $\mathbb{S}^1\times \mathrm{Mo}$, the orientable non-trivial $I$-bundle over $K^2$, $K^2\tilde{\times}I$ or $\bsfour$, a space homeomorphic to a small metric ball of an $\mathbb{S}^2$-soul of an open non-negatively curved Alexandrov space $L(\mathbb{S}^2;4)$ (cf. \cite[Corollary 2.56]{MitYam}).
\item[(iv)] If $\partial B\cong K^2$ then $B$ and $B'$ are homeomorphic to one of $\mathbb{S}^1\tilde{\times} D^2$ the non-orientable $D^2$-bundle over $\mathbb{S}^1$, $K^2\hat{\times}I$ the non-orientable non-trivial $I$-bundle over $K^2$, the space $\bpt$ defined in \cite[Example 1.2]{MitYam}, or $\bptwo$, a space homeomorphic to a small metric ball of an $\mathbb{R}P^2$-soul of an open non-negatively curved Alexandrov space $L(\mathbb{R}P^2;2)$ (cf. \cite[Corollary 2.56]{MitYam}).
\end{itemize}

Finally, if $\dim Y=0$ (cf. \cite[Theorem 1.9]{MitYam}), then for $i$ sufficiently big, $X_i$ is homeomorphic to either a generalized Seifert fiber space $\mathrm{Seif}(Z)$, (where $Z$ is a $2$-dimensional Alexandrov space with $\curv Z\geq 0$), a space appearing in the cases in which $\dim Y=1, 2$ or a non-negatively curved Alexandrov space with finite fundamental group. 

In order to provide information on the homotopy groups of some of the spaces appearing in the previous classification we now recall the celebrated Soul Theorem for Alexandrov spaces due to Perelman \cite[\S 6]{Per2}

\begin{thm}[Soul Theorem]
Let $X$ be a compact Alexandrov space of $\curv\geq 0$ with $\partial X\neq \emptyset$. Then there exists a totally convex, compact subset $S\subset X$, called the soul of $X$ with $\partial S=\emptyset$ which is a strong deformation retract of $X$.
\end{thm} 

The spaces $\bstwo$, $\bsfour$, $\bpt$ and $\bptwo$ admit Alexandrov metrics of $\curv\geq 0$. Therefore the Soul Theorem can be applied. Moreover, given that the soul is a strong deformation retract of the space, in particular we have a homotopy equivalence. Whence, $\pi_k(\bstwo)\cong \pi_k(\mathbb{S}^2)$, $\pi_k(\bsfour)\cong \pi_k(\mathbb{S}^2)$, while $\pi_k(\bpt)=0$ and $\pi_k(\bptwo)\cong \pi_k(\mathbb{R}P^2)$ for all $k\geq 1$.

\section{Proofs}
\label{section:borel_conjecture_collapsing_alexandrov_3_spaces}

We proceed to prove Theorem \ref{thm:collapsed_aspherical_implies_manifold}. As stated in the Introduction, this result readily implies our main result, Theorem \ref{thm:Borel_Alex_collapsing}. 

\begin{proof}[Proof of Theorem \ref{thm:collapsed_aspherical_implies_manifold}]
We will proceed by contradiction. Let us suppose that the result in question does not hold. Then, there exists a sequence of closed, irreducible and aspherical Alexandrov $3$-spaces $\{X_i\}_{i=1}^{\infty}$ of $\curv X_i\geq -1$, satisfying that $\diam X_i\leq D$ and $\vol X_i \stackrel {_{i\to\infty}}{\longrightarrow}0$ which are not homeomorphic to $3$-manifolds. Therefore by Gromov's precompactness Theorem we can assume (possibly passing to a non-relabeled subsequence) that $X_i$ collapses in the Gromov-Hausdorff topology to a closed Alexandrov space $Y$ of dimension $<3$.  We will split the proof in three cases depending on whether $\dim Y= 0,1,2$ and obtain a contradiction in each case. Observe that by our contradiction assumption in the following analysis we will exclude any Alexandrov $3$-spaces appearing in the classification \cite{MitYam} that are homeomorphic to $3$-manifolds. 

\subsection{$2$-dimensional limit space}
\label{case:dimY=2}

In the case that $\dim Y=2$, we need to address two further cases depending on whether $X_i$ contains singular fibers with neighborhoods of type $\bpt$ or not. In the case in which $X_i$ does not contain fibers of type $\bpt$, then by \cite[Corollary 6.2]{GalNun}, the collapse $X_i  \stackrel {_{GH}}{\longrightarrow} Y$ is equivalent to the one obtained by collapsing along the orbits of a local circle action on $X_i$. Therefore, by \cite[Theorem B]{GalNun} $X_i$ is homeomorphic to a connected sum of the form $M\#\Susp(\Real P^2)\# \cdots \# \Susp(\Real P^2)$, where $M$ is a closed $3$-manifold admitting a local circle action. It now follows from \cite[Lemma 5.1]{GalGuiNun} that either $X_i$ is homeomorphic to $M$ or to $\Susp(\Real P^2)$. Since we are assuming $X_i$ is not homeomorphic to a $3$-manifold, we conclude that $X_i$ is homeomorphic to $\Susp(\mathbb{R}P^2)$. However, $\Susp(\Real P^2)$ is not aspherical, as a combination of the suspension isomorphism and the Hurewicz Theorem yields that $\pi_2(\Susp(\Real P^2))$ is isomorphic to $\mathbb{Z}_2$.

We move on to the case in which $X_i$ contains fibers with tubular neighborhoods of the form $\bpt$. Let us consider the branched orientable double cover $\tilde{X}_i$ of $X_i$. We now recall that, by \cite[Corollary 3.2]{GalGuiNun} $X_i$ collapses if and only if the sequence of branched orientable double covers $\tilde{X}_i$ collapses.  Then, by \cite[Corollary 6.2]{GalNun}, the connected sum decomposition in \cite[Theorem B]{GalNun}, and taking into account the orientability of $\tilde{X}_i$ we have that $\tilde{X}_i$ is homeomorphic to a connected sum of the form 
\begin{equation}
\label{eq-connected-sum}
(\csum_{\varphi}\mathbb{S}^2\times \mathbb{S}^1)\#(\csum_{j=1}^n L(\alpha_j,\beta_j))
\end{equation} 
where $L(\alpha_j,\beta_j)$ denotes a lens space determined by the Seifert invariants $(\alpha_j,\beta_j)$ (see \cite[Section 1.7]{Orl}).  

It was proved in \cite[Page 14]{GalGuiNun} that, in this situation, the connected sum \eqref{eq-connected-sum} cannot contain both lens spaces and copies of $\mathbb{S}^2\times \mathbb{S}^1$, that is, either $X_i$ is homeomorphic to a connected sum of lens spaces or to a connected sum of copies of $\mathbb{S}^2\times \mathbb{S}^1$. However, the irreducibility assumption implies, as in \cite[Case 5.6]{GalGuiNun}, that the expression \eqref{eq-connected-sum} cannot contain $\csum_{j=1}^n L(\alpha_j,\beta_j)$ as a connected summand and therefore only copies of $\mathbb{S}^2\times \mathbb{S}^1$ appear in the connected sum \ref{eq-connected-sum}. Therefore $X_i$ is homeomorphic to $\csum_{\varphi}\mathbb{S}^2\times \mathbb{S}^1 $. Furthemore, it follows from the irreducibility of $X_i$ as in \cite[Case 5.7]{GalGuiNun} that $\tilde{X}_i$ is homeomorphic to $\mathbb{S}^2\times \mathbb{S}^1$. Hence, it suffices to consider the case in which $X_i$ is a quotient of $\mathbb{S}^2\times \mathbb{S}^1$ by an orientation reversing involution $\iota: \mathbb{S}^2\times \mathbb{S}^1 \to \mathbb{S}^2\times \mathbb{S}^1$ having only isolated fixed points. 
 
Let us consider $\mathbb{S}^2\times \mathbb{S}^1$ as a subspace of $\mathbb{R}^3\times \mathbb{C}$ and denote its points by $((x,y,z),w)$. The classification \cite{Tol} of involutions on $\mathbb{S}^2\times \mathbb{S}^1$ yields that the involution $\iota_i$ on $\tilde{X}_i$ satisfying that $\tilde{X}_i/\iota_i\cong X_i$ is equivalent to the involution defined by 
\[
((x,y,z),w) \mapsto ((-x,-y,z),\overline{w})
\]
which has four fixed points. The quotient space of this involution is homeomorphic to $\Susp(\mathbb{R}P^2)\#\Susp(\mathbb{R}P^2)$ (see incise (a) of Case (2) in Page 7 of \cite{GalGui}). Then it follows, as in the proof of \cite[Theorem 6.1]{Nun}, that $X_i$ is not aspherical.

\subsection{$1$-dimensional limit space}
\label{ssec:1-dim-limit}

In this case, for sufficiently big $i$, $X_i$ is homeomorphic to a gluing of two pieces $B$ and $B'$ appearing in the classification \cite{MitYam}, along their isometric boundaries. As $X_i$ is not homeomorphic to a manifold, at least one of the following pieces must appear in the decomposition: $\bstwo$, $K_1(\mathbb{R}P^2)$, $\bsfour$, $\bpt$ and $\bptwo$. We will show in this section that every possible space $X_i$ having one of these pieces cannot be aspherical. 

\begin{case}[One of the pieces is $\bstwo$] The possible pieces in this case are $D^3$, $\mathbb{R}P^3\setminus \mathrm{int}D^3$ and $\bstwo$. By \cite[Remark 2.62]{MitYam}, $B(S_2)$ is homeomorphic to $\mathrm{Susp}(\mathbb{R}P^2)\setminus D^3$. Whence the possible spaces arising by gluing with this piece are homeomorphic to $\mathrm{Susp}(\mathbb{R}P^2)$, $\mathbb{R}P^3\# \mathrm{Susp}(\mathbb{R}P^2)$ and $\mathrm{Susp}(\mathbb{R}P^2)\#\mathrm{Susp}(\mathbb{R}P^2)$. All of these spaces have $\mathrm{Susp}(\mathbb{R}P^2)$ as a connected summand, and therefore the arguments of \cite[Theorem 6.1]{Nun} show that they are not aspherical, a contradiction.
\end{case}

\begin{case}[One of the pieces is $K_1(\mathbb{R}P^2)$] The only piece is the closed cone $K_1(\mathbb{R}P^2)$ and therefore, the space $X_i$ in this case is homeomorphic to $\mathrm{Susp}(\mathbb{R}P^2)$. As previously mentioned, this space is not aspherical and then a contradiction is ensued. 
\end{case}

As a first step to deal with the remaining cases, we observe that $H_2(X_i)\neq 0$ whenever $X_i$ contains a piece of the form $\bsfour$ or $\bpt$ a fact that follows from the following result due to Mitsuishi (see \cite[Corollary 5.7]{Mit}). The result is originally stated for the more general class of $NB$-spaces (see \cite[Definition 1.6]{Mit}). For simplicity we restate it here for Alexandrov spaces only.

\begin{thm}[Mitsuishi]
\label{thm:Mitsuishi_result}
Let $X$ be a closed, connected Alexandrov  $n$-space. If $X$ is non-orientable then the torsion subgroup of $H_{n-1}(X;\mathbb{Z})$ is isomorphic to $\mathbb{Z}_2$ and, in particular, $H_{n-1}(X;\mathbb{Z})$ is non-zero.  
\end{thm}

In order to obtain that $\pi_2(X_i)\neq 0$ using the information that $H_2(X_i)\neq 0$ we will use the following classical theorem proved by Hopf in \cite[Theorem a), Page 257]{Hopf}.

\begin{thm}[Hopf]
\label{thm:Hopf}
Let $X$ be a $CW$-complex with finitely many cells. Then, there exists an exact sequence
\begin{equation}
\label{eq:Hopf_Sequence}
\pi_2(X)\to H_2(X)\to H_2(B\pi_1(X);\mathbb{Z})\to 0.
\end{equation}
\end{thm}
Here, $B\pi_1(X)$ denotes a model for the classifying space of the fundamental group $\pi_1(X)$, which is characterized up to homotopy by being an aspherical $CW$-complex having the same fundamental group as $X$. As $\pi_1(X)$ depends on the pieces $B$ and $B'$ we will split the following analysis to go over every possibility. 

\begin{case}[One of the pieces is $\bsfour$] The possible pieces $B$ and $B'$ with $X_i\cong B\cup B'$ in this case are $S^1\times D^2$, $S^1\times \mathrm{Mo}$, $K^2\tilde{\times} I$ and $B(S_4)$. We assume that $B'=B(S_4)$ is fixed.

\begin{subcase}[$B=B(S_4)$]
\label{sc:BS4}
If both pieces of the decomposition of $X_i$ are homeomorphic to $B(S_4)$, then Van Kampen's Theorem readily implies that $\pi_1(X_i)=0$. Moreover, by Theorem \ref{thm:Mitsuishi_result}, $H_2(X_i)\neq 0$. Therefore, by Hurewicz's Theorem, $\pi_2(X_i)\neq 0$, which contradicts the asphericity of $X_i$. 
\end{subcase}

\begin{subcase}[$B=S^1\times\mathrm{Mo}$]
The fundamental group of $X_i$ in this case is $\mathbb{Z}\oplus \mathbb{Z}$, as calculated from Van Kampen's Theorem. A model for the classifying space $B(\mathbb{Z}\oplus \mathbb{Z})$ is the torus $T^2$. Hence, the sequence \eqref{eq:Hopf_Sequence} becomes
\begin{equation}
\label{eq:BS4-S1xMo}
\pi_2(X_i) \rightarrow H_2(X_i)\rightarrow \mathbb{Z}\rightarrow 0.
\end{equation}

Here, we have used the fact that $H_2(T^2)\cong \mathbb{Z}$ (as a consequence of the orientability of $T^2$). Therefore, if $\pi_2(X_i)=0$ we would obtain that $H_2(X_i)\cong \mathbb{Z}$, in particular yielding that $H_2(X_i)$ is torsion-free. This contradicts Theorem \ref{thm:Mitsuishi_result}. Therefore, $\pi_2(X_i)\neq 0$ which is a contradiction to the asphericity of $X_i$. 
\end{subcase}

\begin{subcase}[$B=S^1\times D^2$] 
\label{sc:S1xD2}
In this case, a computation via the Van Kampen's Theorem yields that $\pi_1(X_i)\cong \mathbb{Z}$. Now, a model for $B\mathbb{Z}$ is the circle $S^1$. Hence, Hopf's exact sequence \eqref{eq:Hopf_Sequence} takes the form
\begin{equation}
\label{eq:BS4_s1xD2}
\pi_2(X_i)\rightarrow H_2(X_i)\rightarrow 0\rightarrow 0. 
\end{equation}
Therefore, the morphism $\pi_2(X_i)\to H_2(X_i)\neq 0$ is surjective, implying that $\pi_2(X_i)\neq 0$. This contradicts the assumption that $X_i$ is aspherical. 
\end{subcase}

\begin{subcase}[$B=K^2\tilde{\times} I$]
\label{sc:K2tildeI}
As in the previous cases, using Van Kampen's theorem it follows that the fundamental group of $X_i$ is that of $K^2\tilde{\times} I$. Since $I$ is contractible, $\pi_1(K^2\tilde{\times}I)\cong \pi_1(K^2)\cong \mathbb{Z}\rtimes \mathbb{Z}$. We now note that a model for $B(\mathbb{Z}\rtimes \mathbb{Z})$ is the Klein bottle. Therefore by the non-orientability of $K^2$, the Hopf's sequence \eqref{eq:Hopf_Sequence} becomes
\begin{equation}
\label{eq:BS4_K2tildeI}
\pi_2(X_i)\rightarrow H_2(X_i)\rightarrow 0 \rightarrow 0.
\end{equation}
Whence, as in the previous case, $\pi_2(X_i)\neq 0$.
\end{subcase}
\end{case}

\begin{case}[One of the pieces is $\bpt$]
\label{case:bpt}
In this case the possible pieces taking the role of $B$ and $B'$ are $S^1\tilde{\times} D^2$, $K^2\hat{\times} I$, $\bpt$ and $\bptwo$. We will exclude the space $\bpt\cup \bptwo$ in this case as it will be considered below. Let us observe that by Van Kampen's Theorem, the possible fundamental groups of $X_i$ in this case are the same that appear in the case that one of the pieces is $\bsfour$. Therefore, a contradiction to the asphericity of $X_i$  is obtained for $\bpt\cup \bpt$ as in Case \ref{sc:BS4}, for $\bpt K^2\hat{\times}I$ as in Case \ref{sc:K2tildeI}, and for $\bpt \cup S^1\tilde{\times} D^2$ as in Case \eqref{sc:S1xD2}. 
\end{case}

We now address the remaining case in which one of the pieces in the decomposition of $X_i$ is $\bptwo$. 

\begin{case}[One of the pieces is $\bptwo$]
Under this assumption, the possible pieces $B$ and $B'$ forming $X_i$ are $S^1\tilde{\times} D^2$, $K^2\hat{\times} I$, $\bpt$ and $\bptwo$. As some possibilities overlap with the previous Case \eqref{case:bpt}, we only consider the spaces $S^1\tilde{\times} D^2 \cup \bptwo$, $K^2\hat{\times} I \cup \bptwo$, $\bpt \cup \bptwo$ and $\bptwo \cup \bptwo$ here. To address the question of asphericity of these spaces we apply the following result (see \cite[Lemma 4.1]{Luck})

\begin{thm}
\label{thm:aspherical_fund_group_no_torsion}
The fundamental group of an aspherical finite-dimensional $CW$-complex is torsion-free.
\end{thm}
 
An analysis via Van Kampen's Theorem yields that the possible fundamental groups of $X_i$ in this case are $\mathbb{Z}\ast_{\pi_1(K^2)} \mathbb{Z}_2$, $(\mathbb{Z}\rtimes \mathbb{Z})\ast_{\pi_1(K^2)} \mathbb{Z}_2$, $\mathbb{Z}_2$ and $\mathbb{Z}_2\ast_{\pi_1(K^2)} \mathbb{Z}_2$. It is immediate to check that these groups have non-zero torsion, as at least one of the factors in the amalgamated products has non-zero torsion.  Therefore Theorem \ref{thm:aspherical_fund_group_no_torsion} yields that $X_i$ cannot be aspherial, a contradiction. 
\end{case}

\subsection{$0$-dimensional limit space}
In this case, if $X_i$ is a generalized Seifert fibered space then the contradiction is obtained as in Section \ref{case:dimY=2}. If $X_i$ is homeomorphic to a space appearing in the $1$-dimensional limit case, then the contradiction is obtained as in Section \ref{ssec:1-dim-limit}. The remaining cases are non-negatively curved (non-manifold) Alexandrov $3$-spaces with finite fundamental group. In these cases,  if $\pi_1(X_i)=0$, Theorem \ref{thm:Mitsuishi_result}, implies that $X_i$ is not aspherical. Furthermore, if $\pi_1(X_i)$ is non-trivial then \ref{thm:aspherical_fund_group_no_torsion} yields that $X_i$ is not aspherical. Hence, in every case we obtain a contradiction to asphericity and the result is settled.   
\end{proof}

\begin{rem}
\label{conjecture}
In light of Theorem \ref{thm:collapsed_aspherical_implies_manifold}, a natural conjecture would be that a closed, aspherical Alexandrov $3$-space $X$ is homeomorphic to a $3$-manifold. This is indeed the case whenever $X$ is simply-connected as a consequence of incise (2) of \cite[Corollary 5.7]{Mit} and the Hurewicz Theorem. 
\end{rem}


\section{Questions  on fundamental  groups  and  actions  on  Alexandrov 3-Spaces}\label{section:groups_spaces_questions}

A  lot  of  effort  has  been  devoted  in geometric  topology  to   the  development of  characterization of   fundamental groups of  manifolds  and  spaces  which  are  topologically  rigid. On  the  other  hand, a  similar  fruitful  effort  has  been  devoted  to  the  characterization of  groups which  act on geometrically defined classes of  manifolds.  

We  will  now  consider  a series  of  questions inspired  by  the study  of  topological  rigidity of manifolds  and  their  extrapolation  to (possibly  singular) Alexandrov $3$-spaces.

These questions evolve  from  topological  or  geometric rigidity    results such   as  the  Borel  Conjecture for  the  fundamental group $G$  of  an  aspherical manifold, into the characterization  of  fundamental  groups of   such  spaces  through concepts  of  geometric  group  theory  or group  cohomology. 

Specifically  related  to  group  cohomology, it  is  an  open  conjecture originally  posed    by  Wall  \cite{wallpoincare3} that  every  poincar\'e  duality  group  of  dimension $3$ is  the  fundamental  group of  a  three  dimensional  manifold. 

See  \cite{lueckferryweinberger},  \cite{hillmannthreedimensionalpoincare} for a  modern discussions on the  subject. Here, we  present  the  following  question: 
 
\begin{question}\label{prob:poincare}
Let  $G$  be   a  Poincar\'e  duality group  which is  the  fundamental  group  of  an  aspherical Alexandrov $3$-space. Is  it  the  fundamental group of an orientable three dimensional manifold? 
 \end{question} 
 
Question  \ref{prob:poincare} follows  readily  from Conjecture \ref{conj:alex-manifold}.

 In  the  direction  of characterizations  of  groups acting  on   manifolds, it  is  a classical  result  by Wall \cite{wallpoincare3},  that the  finite  groups  acting on three-dimensional  Poincar\'e  complexes  have  periodic  cohomology  of  period  $4$. It  is  known that the symmetric group on three  letters  $\Sigma_3$  cannot  be  realized  by  any  honest manifold \cite{milnorfreeactionsspheres}.

\begin{question}\label{prob:groups}
 Which finite  groups  act  by  homeomorphisms on Alexandrov $3$-spaces?   
\end{question}

The  study  of  geometric  and  large  scale  geometric  properties  of  fundamental   groups  of  three  dimensional  manifolds in connection  with topological  rigidity has been  oriented  in  recent  times   to  the characterization of the  map involved  in  the homotopy  equivalence referred  to in the  statement. Consider  as  an  example  the  following question.

\begin{question}
Let  $f: M\to N$  be  a  map   between three  dimensional, aspherical  manifolds  with  boundary  inducing  an  epimorphism  in fundamental groups. Under  which  conditions  is  $f$ homotopic  to  a  homeomorphism.? 
\end{question}

This problem  has  been studied  using  simplicial  volume, specifically  degree  theorems,  and  more  recently,  Agol's solution  to  the  virtually  fibering  conjecture  by  Boileau and  Friedl \cite{boileaufriedlepi}. 
 
(The  epimorphic  condition  comes  from  the  fact  that  a  degree  one  map  induces  such  an  epimorphism due  to  Poincar\'e  Duality  and  the  loop  theorem). 
\subsection{Consequences  of  Agol's virtually fibering Theorem}
The  following  theorem  was proved  by Boileau  and Friedl, \cite{boileaufriedlepi}. 
\begin{theorem}
Let $f: M\to N$ be  a  proper map  between aspherical  manifolds  with  either  toroidal  or  empty boundary. Assume  that  $N$  is  not  a  closed graph manifold, and that  $f$ induces  an epimorphism  on the  fundamental  group.Then  $f$  is  homotopic  to  a  homeomorphism if  any  of  the  following  two  conditions are  met: 
\begin{itemize}
\item For  each  $H$ finite  index, subnormal  subgroup  of $\pi_1(N)$, the  ranks  of  $H$ and  $f^{-1}(H)$ agree.  
\item For  each finite  cover $\widetilde{N}$  of $N$,  the  Heegard genus of $\widetilde{N}$  and $\widetilde{M}$ agree.  
\end{itemize}
\end{theorem}

\begin{question}
Let $f: M\to N$  be  an  oriented  map  between oriented, aspherical  Alexandrov  spaces. What  are  the  conditions  on $f$, $M$, and  $N$ for $f$ be  homotopic  to  homemomorphism?    
\end{question}


\bibliographystyle{amsplain}

\begin{thebibliography}{10}

\bibitem{BarLuc} A.~Bartels and W.~L\"uck, \textit{The Borel conjecture for hyperbolic and $\mathrm{CAT}(0)$-groups}, Ann. of Math. (2) 175 (2012), no. 2, 631--689. 

\bibitem{boileaufriedlepi} M.~Boileau and S.~Friedl, \textit{Epimorphisms of 3-manifold groups.} Q. J. Math., 69(3):931--942, 2018.


\bibitem{BurBurIva} D.~Burago,  Y.~Burago and S.~Ivanov, \textit{A course in metric geometry}, Graduate Studies in Mathematics, vol.~33, American Mathematical Society, Providence, RI, 2001.

\bibitem{DenGalGuiMunn} Q.~Deng, F.~Galaz-Garc\'ia, L.~Guijarro and M.~Munn, \textit{Three-Dimensional Alexandrov Spaces with Positive or Nonnegative Ricci Curvature}, Potential Anal. 48 (2018), no. 2, 223--238.

\bibitem{lueckferryweinberger}
S.~Ferry, W.~L\"uck, and S.~Weinberger. \textit{On the stable Cannon conjecture}, preprint (2018) {\tt arXiv:1804.00738 [math.GT]}

\bibitem{GalSurvey} F.~Galaz-Garc\'ia, \textit{A glance at three-dimensional Alexandrov spaces}, Front. Math. China 11 (2016), no. 5, 1189--1206.

\bibitem{GalGui} F.~Galaz-Garc\'ia and L.~Guijarro, \textit{On three-dimensional Alexandrov spaces}, Int. Math. Res. Not. IMRN 2015, no. 14, 5560--5576.

\bibitem{hillmannthreedimensionalpoincare} J.~Hillmann, \textit{ Some questions on subgroups of 3-dimensional poincar\'e duality  groups}, preprint (2016) {\tt arXiv:1608.01407 [math.GR]}

\bibitem{GalGuiNun} F.~Galaz-Garc\'ia, L.~Guijarro and J.~N\'u\~{n}ez-Zimbr\'on, \textit{Sufficiently collapsed irreducible Alexandrov $3$-spaces are geometric}, to appear Indiana U. Math. J., preprint (2017) {\tt arXiv:1709.02336 [math.DG]} 

\bibitem{GalNun} F.~Galaz-Garc\'ia and J.~N\'u\~{n}ez-Zimbr\'on, \textit{Three-dimensional Alexandrov spaces with local isometric circle actions}, to appear Kyoto J. Math., preprint (2017) {\tt 	arXiv:1702.06080 [math.DG]}

\bibitem{GroWil} K.~Grove and B.~Wilking, \textit{A knot characterization and $1$-connected nonnegatively curved $4$-manifolds with circle symmetry}, Geom. Topol. 18 (2014), no. 5, 3091--3110.

\bibitem{Hopf} H.~Hopf, \textit{Fundamentalgruppe und zweite Bettische Gruppe},  Comment. Math. Helv. 14, (1942). 257--309.


\bibitem{Luck} W.~L\"uck, \textit{Survey on aspherical manifolds}, European Congress of Mathematics, 53--82, Eur. Math. Soc., Z\"urich, 2010.
\bibitem{LuRe} W.~L\"uck and H. Reich, \textit{The Farrell-Jones  and  the Baum-Connes Conjectures in {K}- and {L}- Theory}, In {\em Handbook of K-Theory}, Volume  2,  pages 703--842. Springer Verlag, 2005.  

\bibitem{milnorfreeactionsspheres} J.~Milnor. \textit{Groups which act on {$S^n$} without fixed points.} {Amer. J. Math.}, 79:623--630, 1957.

\bibitem{Mit} A.~Mitsuishi, \textit{Orientability and fundamental classes of Alexandrov spaces with applications}, preprint (2016) {\tt arXiv:1610.08024 [math.MG]}

\bibitem{MitYam} A.~Mitsuishi and T.~Yamaguchi, \textit{Collapsing three-dimensional closed Alexandrov spaces with a lower curvature bound}, Trans. Amer. Math. Soc. 367 (2015), no. 4, 2339--2410.

\bibitem{Nun} J.~N\'u\~{n}ez-Zimbr\'{o}n, \textit{Closed three-dimensional Alexandrov spaces with isometric circle actions}, Tohoku Math. J. (2) 70 (2018), no. 2, 267--284. 

\bibitem{Orl} P.~Orlik, \textit{Seifert manifolds}, Lecture Notes in Mathematics, Vol. 291. Springer-Verlag, Berlin-New York, 1972.

\bibitem{Per2} G.~Perelman, \textit{A.~D.~Alexandrov’s spaces with curvatures bounded from below II}, Preprint.

\bibitem{Per} G.~Perelman, \textit{Elements of Morse theory on Aleksandrov spaces}, St. Petersburg Math. J. 5 (1994), 205--213.

\bibitem{Por} J.~Porti, \textit{Geometrization of three manifolds and Perelman's proof}, Rev. R. Acad. Cienc. Exactas F\'is. Nat. Ser. A Math. RACSAM 102, no.1, (2008), 101–125.

\bibitem{Star} E. ~ Stark, \textit{Topological  Rigidity  fails  for  quotients  of  the  Davis  Complex}, preprint (2016) {\tt arXiv:1610.08699 [math.GT]} 

\bibitem{TakYok} Y.~Takeuchi and M.~Yokoyama, \textit{Waldhausen's classification theorem for $3$-orbifolds}, Kyushu J. Math. 54 (2000), no. 2, 371--401.

\bibitem{Tol} J.~L.~Tollefson, \textit{Involutions on $S\sp{1}\times S\sp{2}$ and other $3$-manifolds}, Trans. Amer. Math. Soc. 183 (1973), 139--152.

\bibitem{wallpoincare3} C.~T.~C. Wall, \textit{Poincar\'{e} duality in dimension 3}, Proceedings of the Casson Fest, 1--26, Geom. Topol. Monogr., 7, Geom. Topol. Publ., Coventry, 2004.

\end{thebibliography}



\end{document}